\documentclass{amsart}

\usepackage{amsmath, amsfonts, amssymb, amsthm, mathtools}
\usepackage{enumerate}  
\usepackage{stmaryrd}
\usepackage{tikz-cd}

\usepackage{hyperref}
\hypersetup{
	colorlinks = true,
	citecolor = blue,
  	linkcolor  = blue,
  	urlcolor   = blue
}

\usepackage{mathtools}
\usepackage{todonotes}

\numberwithin{equation}{section}

\theoremstyle{plain}
\newtheorem{theorem}[equation]{Theorem}
\newtheorem{proposition}[equation]{Proposition}
\newtheorem{lemma}[equation]{Lemma} 
\newtheorem{corollary}[equation]{Corollary}

\theoremstyle{definition}

 \newtheorem{example}[equation]{Example}
 \newtheorem{question}[equation]{Question}
\newtheorem*{ack}{Acknowledgements}

\theoremstyle{remark}
\newtheorem{remark}[equation]{Remark}

\newcommand{\dbcat}[1]{\mathrm{D}^b(\mathrm{mod}\,{#1})}

\newcommand{\bbZ}{\mathbb{Z}}
\newcommand{\bbN}{\mathbb{N}}

\newcommand{\bass}{\operatorname{I}}
\newcommand{\betti}{\operatorname{P}}

\newcommand{\dcat}{\mathbf{D}}
\newcommand{\depth}{\operatorname{depth}}
\newcommand{\eee}{\operatorname{E}}
\newcommand{\ext}{\operatorname{Ext}}

\newcommand{\hh}{\operatorname{H}}
\newcommand{\Hom}{\operatorname{Hom}}
\newcommand{\image}{\operatorname{image}}
\newcommand{\injcurv}{\operatorname{inj\,curv}}
\newcommand{\injdim}{\operatorname{inj\,dim}}
\newcommand{\rgam}{\mathbf{R}\varGamma}
\newcommand{\lescot}{\operatorname{W}}
\newcommand{\lotimes}{\otimes^{\mathbf L}}
\newcommand{\rmod}{\operatorname{mod}}
\newcommand{\projdim}{\operatorname{proj\,dim}}
\newcommand{\rank}{\mathrm{rank}}
\newcommand{\rhom}{\mathbf{R}\mathrm{Hom}}
\newcommand{\syz}{\operatorname{\Omega}\!}
\newcommand{\tee}{\operatorname{T}}
\newcommand{\tor}{\operatorname{Tor}}

\newcommand{\fm}{\mathfrak{m}}

\title[Unstable elements in cohomology]{Unstable elements in cohomology and \\ a question  of Lescot}

\author[Iyengar]{Srikanth B. Iyengar}
\address{Department of Mathematics, University of Utah, Salt Lake City, UT, USA}
\email{srikanth.b.iyengar@utah.edu}

\author[Maitra]{Sarasij Maitra}
\address{Department of Mathematics, University of Utah, Salt Lake City, UT, USA}
\email{maitra@math.utah.edu}

\author[Tribone]{Tim Tribone}
\address{Department of Mathematics, University of Utah, Salt Lake City, UT, USA} \email{tim.tribone@utah.edu}

\subjclass[2010]{13D02, 13D07}
\keywords{Bass numbers, coherent Ext-algebras,  stable cohomology}

\begin{document}
\begin{abstract}
In his work  on the Bass series of syzygy modules of modules over a commutative noetherian local ring $R$, Lescot introduces a numerical invariant, denoted $\sigma(R)$, and asks whether it is finite for any $R$. He proves that this is so when $R$ is Gorenstein or Golod. In the present work many new classes of rings $R$ for which $\sigma(R)$ is finite are identified. The new insight is that $\sigma(R)$ is related to the natural map from the usual cohomology of the module to its stable cohomology, which permits the use of multiplicative structures to study the question of finiteness of $\sigma(R)$.
	\end{abstract}

\date{\today}

\maketitle

\setcounter{tocdepth}{1}
\tableofcontents

\section{Introduction}
This work grew out of a problem concerning the growth of Bass numbers of modules over local rings.  Fix a noetherian, commutative, local ring $(R,\fm,k)$, where $\fm$ is the maximal ideal, and $k$ is the residue field. The \emph{Bass numbers} of a finitely generated $R$-module $M$ are the integers
\[
\mu^i(M)\coloneqq \rank_k\ext^i_R(k,M) \quad \text{for $i\in \bbZ$.}
\]
Thus, if $M\xrightarrow{\sim} I$ is the minimal injective resolution of $M$ over $R$, then $\mu^i(M)$ is the number of copies of $E(k)$, the injective hull of $k$, occurring in $I^i$.  The \emph{Betti numbers} of $M$ are the integers
\[
\beta_i(M)\coloneqq \rank_k\tor^R_i(k,M) \quad \text{for $i\in \bbZ$.}
\]
Thus $\beta_i(M)=\rank_R(F_i)$ where $F\xrightarrow{\sim}M$ is the minimal free resolution of $M$. A starting point of this project is the following:

\begin{question}
\label{qu:bass-M}
Fix an integer $n\ge 1$.  How are the Bass numbers of $\Omega^nM$ related to those of $M$?
\end{question}
Here $\Omega^nM$ denotes the $n$th syzygy module of $M$. The analogue of this question concerning the Betti numbers of $M$ is easily resolved:
\[
\beta_i(\Omega^nM) = \beta_{n+i}(M) \quad\text{for all $i\ge 0$.}
\]
This is because if $F$ is the minimal free resolution of $M$ the truncation $F_{\geqslant n}$ is the minimal free resolution of $\Omega^nM$. It can also be deduced by applying $-\otimes_R k$ to the exact sequences of $R$-modules defining the syzygy modules:
\[
0\longrightarrow \Omega^{n}M \longrightarrow F_n \longrightarrow \Omega^{n-1}M\longrightarrow 0
\]
keeping in mind that $\tor^R_i(k,R)=0$ for $i\ge 1$. A similar argument yields that if ring $R$ is Gorenstein, equivalently, if $\ext^i_R(k,R)=0$ for $i\ne \dim R$,  then
\[
\mu_i(\Omega^nM) = \mu_i(M)\quad\text{for $i\ge \dim R+1$.}
\]
Thus Question~\ref{qu:bass-M} is of interest only when $R$ is \emph{not} Gorenstein, because  then there is no such simple relationship between the Bass numbers of $\Omega^nM$ and the Bass numbers of $M$, for a general $R$-module $M$.

The case $M=k$, the residue field of $R$, is already of interest.  In \cite{Lescot:1986}, Lescot expresses the Bass series (that is to say, the generating series of the Bass numbers) of $\Omega^nk$ in terms of the Bass series of $k$ and of $R$.  It is immediate from this result that Bass numbers of $\Omega^nk$ and of $k$ grow at the same rate; in fact, the growth is exponential when $R$ is not complete intersection; see \cite{Avramov:1998a}. Lescot~\cite{Lescot:1986} also proves that
that each non-zero \emph{direct summand} of $\Omega^nk$ has infinite injective dimension. This leads us to ask:

\begin{question}
\label{qu:bass}
Let $N$ be a nonzero direct summand of $\Omega^n_R(k)$ for some $n\ge 1$. What can one say about growth of the Bass numbers of $N$? Is it exponential, and if so, what is its order?
\end{question}

We know the answers to this question (yes, the growth is exponential and of the same order as that of $k$) when $R$ is Gorenstein or Golod, but open in general. Question~\ref{qu:bass} is also suggested by a result of Avramov's~\cite{Avramov:1996} that the Betti numbers of $N$ grow at the same rate as the Betti numbers of $k$, which are also the Bass numbers of $k$. However, we have not been able to adapt Avramov's proof to treat Bass numbers.  It appears to be  useful to rather go back to Lescot's work~\cite{Lescot:1986} to search for clues to a solution to the question above.  Lescot considers, for any $R$-module $M$, the map 
\begin{equation*}
\theta(M)\colon  \tor^R(k,M)\otimes_k \tor^R(k,M^\vee)\longrightarrow \tor^R(k,E)
\end{equation*}
induced by the evaluation map $M\otimes_R M^\vee \to E$, where $E$ is the injective hull of $k$ and $M^\vee = \Hom_R(M,E)$ is the Matlis dual of $M$; see ~\cite{Bruns/Herzog:1998a}. Following Lescot let $W(M)$ denote the image of the map above. Its relevance to Question~\ref{qu:bass-M} is that when $W(M)=0$ one can expresses the Bass series of $\Omega^nM$ in terms of the Bass series of $M$ and that of $R$; see \cite{Lescot:1986} and also Theorem~\ref{th:bass-series-formula}. Lescot~\cite{Lescot:1983} proves that $W(k)=0$, and this leads to his result on the Bass numbers of $\Omega^nk$ mentioned above.  These results also explains why the following invariant, introduced by Lescot in \cite{Lescot:1986}, is of interest:
\[
\sigma(R)\coloneqq \inf\{n\ge 0\mid W(\Omega^nM)=0 \text{ for all finitely generated modules $M$}\}
\]
as is the following question posed by Lescot~\cite{Lescot:1986}.

\begin{question}
    Is $\sigma(R)<\infty$ for any local ring $R$?
\end{question}

A surprising aspect of this question: It is not even clear that $W(\Omega^nM)=0$ for $n\gg 0$ for a given $M$! Lescot gives no hint to why he expect this to hold, leave alone the assertion that there is bound on $n$ independent of $M$. In \cite{Lescot:1986}, he verifies that the question has an affirmative answer for Gorenstein rings and Golod rings; there has been no further progress on it, as far as we know.  

Our first step is to recast the conjecture, using Matlis duality, in terms of the following map 
\begin{equation*}
\eta(M)\colon  \ext_R(k,R)\otimes_k \tor^R(k,M)\longrightarrow \ext_R(k,M)
\end{equation*}
that is adjoint to the map $\theta(M)$.  Let $U(M)$ denote the image of the map $\eta(M)$ defined above. Because the pairings in question are adjoint to each other, $W(M)=0$ if and only if $U(M)=0$, so $\sigma(R)$ can be expressed in terms of the vanishing of $U(\Omega^nM)$. Our interest in $U(M)$ is that it is the space of unstable elements in $\ext_R(k,M)$, in the following sense: There is an exact sequence
\[
\ext_R(k,R)\otimes_k \tor^R(k,M)\xrightarrow{\ \eta(M)\ } \ext_R(k,M)\longrightarrow \widehat{\ext}_R(k,M)
\]
where the object on the right is the stable cohomology of the pair $(k,M)$; see, for instance, \cite{Avramov/Veliche:2007}. From this perspective, Lescot's conjecture that  $\sigma(R)$ is finite becomes the assertion that there is an integer $n\ge 0$ such that for any finitely generated $R$-module $M$ the map
\[
\ext_R(k,\Omega^nM)\longrightarrow \widehat{\ext}_R(k,\Omega^n M)
\]
is one-to-one. This is a surprising and unexpected (to us) claim about stability of cohomology classes of modules over local rings even outside the realm of Gorenstein rings, where it is known. 

The interpretation $U(M)$ in terms of the map $\eta(M)$ also clarifies the relationship between the condition $W(M)=0$, equivalently, $U(M)=0$, and the computation of the Bass series of syzygy modules of $M$; see the proof of Theorem~\ref{th:bass-series-formula}. Another benefit is that the Ext-algebra $\ext_R(k,k)$ acts on the source and the target of $\eta(M)$, and the map is equivariant with respect to this action. The action on $\ext_R(k,M)$ is the obvious one, obtained by either splicing exact sequences, in the Yoneda interpretation of Ext, or by composition, if one identifies $\ext^n_R(X,Y)$ with morphisms $X\to \Sigma^nY$ in the derived category. The $k$-algebra $\ext_R(k,k)$ also acts $\tor^R(k,M)$ on the left. Intertwining the two actions using the coproduct on $\ext_R(k,k)$ gives the $\ext_R(k,k)$-action on the source of the map $\eta(M)$.
Here is a first indication that bringing in these multiplicative structures is useful: 

\medskip

\emph{
As a module over $\ext_R(k,k)$, if $\ext_R(k,R)$ is generated by elements in degrees $\leq s$,  then $\sigma(R)\le s+1$.
} 

\medskip

See Proposition~\ref{pr:fin-gen}. This observation gives a uniform explanation of Lescot's result that $\sigma(R)$ is finite for Gorenstein rings and Golod rings: For Gorenstein rings $\ext_R(k,R)$ is concentrated in one degree, so it is trivially finitely generated, whereas for Golod rings it is a straightforward calculation to check finite generation; see Roos~\cite{Roos:1978s} and also Proposition~\ref{pr:ggolod}. 

We have been able to identify many other classes of rings $R$ with the desired property. These include rings that are a Golod map away from a complete intersection and local rings $R$ of small co-depth. In fact, we know no rings $R$ for which the finite generation fails, leading us to ask:

\medskip

\emph{
Is $\ext_R(k,R)$ always finitely generated as a module over $\ext_R(k,k)$?
}

\medskip

A positive answer would provide a structural basis for Lescot's conjecture. When $R$ is not complete intersection, the $k$-algebra $\ext_R(k,k)$ is quite complicated; for instance, it is the universal enveloping algebra of an $\bbN$-graded Lie algebra that is non-zero in each degree; what is more the ranks of the graded pieces grow exponentially; see \cite{Avramov:1998a}. For this reason, it appears difficult to tackle the problem above directly. A more promising line of attack is provided by the following result:

\medskip

\emph{
When the $k$-algebra $\ext_R(k,k)$ is coherent, for any finitely generated $R$-module $M$ the $\ext_R(k,k)$-module $\ext_R(k,M)$ is coherent and hence $\sigma(R)<\infty$.
}

\medskip

This is proved by Roos~\cite{Roos:1978s}; see also Proposition~\ref{pr:coherent}. This leads us naturally to the problem of identifying new families of local rings $R$ for which the $k$-algebra $\ext_R(k,k)$ is coherent.

It has been recognized for long that coherence is an important property, especially in the context of general associative rings, but there are not that many techniques to check this property; see, for instance, the first paragraph of \cite{Bonda/Zhdanovskiy:2015}. The coherence of Ext-algebras of local rings is investigated by Roos in \cite{Roos:1982a}; for more recent work, see Gelinas~\cite{Gelinas:2018}.  There are rings $R$ such that $\ext_R(k,k)$ is \emph{not} coherent; see Example~\ref{ex:not-coherent}. 

To summarize the discussion on Lescot's questions: The family of rings $R$ for which the $\ext_R(k,k)$-module $\ext_R(k,R)$ is finitely generated, and hence $\sigma(R)$ is finite, contains the following:
    \begin{enumerate}[\quad\rm(1)]
        \item
        Gorenstein rings;
        \item 
        Golod rings and generalized Golod rings;
        \item
        Absolutely Koszul algebras;
        \item 
        Rings $R$ with $\mathrm{edim}\, R - \depth R \le 3$;
        \item
        Veronese subrings of polynomial rings.
    \end{enumerate}
    Justifications for these claims are in Section~\ref{se:products}. The case where $R$ is Golod or Gorenstein is already in \cite{Lescot:1986}; the rest are new. Moreover, under mild hypotheses, a finite tensor product of rings of the type above is in the family and   the family is closed under descent along finite Gorenstein maps; see Remark~\ref{re:tensor-product} and Theorem~\ref{th:gor-map}.

The focus of this work is on commutative rings, but Lescot's question can be formulated in a broader context. For instance, if $A$ is an Artin algebra, with  maximal semisimple quotient $k$, one can wonder about the stability properties of the kernel of the map $\ext_A(k,M)\to \widehat{\ext}_A(k,M)$, with respect to the syzygies of $M$, as above. This is interesting because $\widehat{\ext}_A(k,M)$ is the graded module of morphisms from $k$ to $M$ in the stable derived category, also known as the singularity category, of $A$. More generally, one can consider this question for semilocal Noether algebras, which encompasses commutative local rings and Artin algebra. It seems plausible that many of the arguments in our work they carry over to this context.


\begin{ack}
Our debt to Lescot's beautiful paper~\cite{Lescot:1986} is obvious from the number of references to it.  Our main contribution is to interpret Lescot's work in terms of stable cohomology---which makes transparent the connection between the invariants $W(R)$ and $\sigma(R)$ introduced in his work and the problem of computing Bass numbers of syzygy modules---and to leverage multiplicative structures on Ext-modules to identify many new families of rings for which his question on $\sigma(R)$ has a positive answer.

We thank Luchezar Avramov and Aldo Conca for helpful conversations regarding the material presented here.

During the course of this project, SBI was partly supported by the National Science Foundation grant DMS-2001368. Part of this work was done while SBI was at the Simons Laufer Mathematical Sciences Institute in Berkeley, California, in Spring 2024,  partly supported by NSF grant No.\ DMS-1928930 and by the Alfred P.\ Sloan Foundation grant G-2021-16778. The second author was supported partially by Project No. 51006801 - American Mathematical Society-Simons Travel Grant.
\end{ack}

\section{Preliminaries}
\label{se:prelims}
Throughout $(R,\fm,k)$ denotes a commutative noetherian local ring with maximal ideal $\fm$ and residue field $k$. We write $\dcat(R)$ for the (full) derived category of $R$-modules, viewed as a triangulated category with suspension $\Sigma$, and $\dbcat R$ for its full subcategory of complexes $M$ whose homology $R$-module, $\hh(M)$,  is finitely generated; that is to say, $\hh_i(M)$ is finitely generated for each $i$ and equal to zero when $|i|\gg 0$.

Some of the graded modules we have to deal with, like $\tor^R(M,N)$, have a natural lower-grading and others, like $\ext_R(M,N)$, have a natural upper-grading. It is expedient to thus assume that each graded module has both gradings, related by $V^i=V_{-i}$ for all $i$.  For a graded-module $V$ we set 
\[
\sup V_* = \sup\{i\mid V_i\ne 0\} \quad\text{and} \quad \inf V_* = \inf\{i\mid V_i\ne 0\}\,.
\]
With this convention the homology of any $R$-complex $M$ has both an upper and a lower grading and there are equalities
\[
\inf \hh^*(M) = - \sup \hh_*(M) \quad\text{and}\quad \sup \hh^*(M) = -\inf \hh_*(M)\,.
\]
Much of this work involves the (covariant) functors
\begin{equation}
    \label{eq:etfunctors}
\tee^R(M)\coloneqq \tor^R(k,M)\quad\text{and}\quad \eee_R(M) \coloneqq \ext_R(k,M)
\end{equation}
from $\dcat(R)$ to graded $k$-vector-spaces. We drop the ring from the notation when it is clear from the context. For $M$ in $\dbcat R$, the graded $k$-vector-space $\tee(M)$ is degree-wise finite---meaning that $\rank_k \tee_i(M)$ is finite for each $i$---and equal to $0$ for $i\ll 0$. Also $\eee(M)$ is degree-wise finite with $\eee^i(M)=0$ for $i\ll 0$.

\begin{lemma}
\label{le:inf}
For $M$ in $\dbcat R$ one has $\inf \tee_*(M)=\inf \hh_*(M)$. Moreover $\tee(M)=0$ if and only if $M\cong 0$, if and only if $\eee(M)=0$.
\end{lemma}

The number  $\inf \eee^*(M)$ is, by definition, the \emph{depth} of $M$; see \cite{Foxby:1979a} and \cite{Foxby/Iyengar:2003}.

\begin{proof}
For $a=\inf\hh_*(M)$ it is clear that one has
 \[
 \tor^R_i(k,M) \cong 
 \begin{cases}
     0& \text{for $i<a$} \\
     k\otimes_R\hh_a(M) & \text{for $i=a$.}
 \end{cases}
 \]
Since the $R$-module $\hh_a(M)$ is finitely generated, Nakayama's Lemma yields that $\tor^R_a(k,M)\ne 0$. This justifies the stated equality and also the claim that $\tee(M)=0$ if and only if $M\cong 0$.  The statement that this happens precisely when $\eee(M)=0$ follows, for example, from  \cite[Proposition~2.8]{Foxby:1979a}.
\end{proof}

\subsection*{Syzygies}
The $n$'th syzygy module of a finitely generated $R$-module $M$ is denoted $\syz^nM$. More generally, given an $R$-complex $M$ in $\dbcat R$ and integer $n$, we write $\syz^nM$ for the \emph{$n$th syzygy complex} of $M$ in the sense of \cite[Section~1]{Avramov/Iyengar:2007a}. Namely, take the minimal free resolution $F$ of $M$ and set
\[
\syz^nM = \Sigma^{-n}F_{\geqslant n}\,.
\]
Since the minimal free resolutions of $M$ are isomorphic as $R$-complexes, the syzygy complexes are independent, again up to an isomorphism of $R$-complexes, of the choice of $F$.  The canonical projection $F\to F_{\geqslant n}$ gives a morphism in $\dbcat R$:
\begin{equation}
    \label{eq:smap}
s_n M\colon M\longrightarrow \Sigma^n\syz^nM\,.
\end{equation}
The observation below is easy to verify.

\begin{lemma}
\label{le:tee-syz}
One has $\tee_i(\Sigma^n\syz^nM)=0$ for $i<n$, and the map $\tee_i(s_nM)$ is an isomorphism for $i\ge n$. \qed
\end{lemma}

\subsection*{Local duality}
Assume that $R$ has a dualizing complex, $\omega_R$, which we take to be a bounded  complex of injective $R$-modules, with $\hh_i(\omega_R)$ finitely generated in each $i$, and normalized so that $\varGamma_{\fm}(\omega_R)$ is the injective hull of $k$. For $M\in \dbcat R$ set
   \[
   M^\dagger \coloneqq \Hom_R(M,\omega_R)\,.
   \]
See \cite[\href{https://stacks.math.columbia.edu/tag/0A7M}{Section 0A7M}]{stacks-project} for basic facts on dualizing complexes. Here is a key one:
\begin{equation}
\label{eq:depth-dim}
\sup \hh_*(M^\dagger) = \dim_RM \quad \text{and}\quad \inf \hh_*(M^\dagger) = \depth_RM\,,    
\end{equation}
where the notion of the dimension  of a complex is as in ~\cite{Foxby:1979a}.  Given the local duality theorem \cite[\href{https://stacks.math.columbia.edu/tag/0A81}{Section 0A81}]{stacks-project}, one could as well define them via  the equalities above.

\subsection*{External Products}
For $R$-modules $L$ and $M$ the natural maps
\[
\begin{gathered}
   \Hom_R(k,L)\otimes_k(k\otimes_RM)\xrightarrow{\ \cong\ }\Hom_R(k,L)\otimes_RM \longrightarrow \Hom_R(k,L\otimes_RM) \\
   f \otimes (x\otimes m) \mapsto \big[y\mapsto f(xy)\otimes m\big]
   \end{gathered}
\]
induce natural  maps in the derived category
\begin{equation*}
\begin{gathered}
       \rhom_R(k,R)\lotimes_k(k\lotimes_RM)\xrightarrow{\ \cong\ }\rhom_R(k,R)\lotimes_RM \longrightarrow \rhom_R(k,M) 
   \end{gathered}
\end{equation*}
In homology, the composition of the maps above yields  the  map
\begin{equation}
\label{eq:eta-defn}
\eta(M) \colon \eee(R)\otimes_k \tee(M)\longrightarrow \eee(M)
\end{equation}
 of graded $k$-vectorspaces that is natural in $M$. This map is interesting because the object on the left is the unstable Ext of the pair $(k,M)$, and the map fits into a long exact sequence relating unstable Ext, the usual Ext, and stable Ext modules. This is explained further below. First we record the following well-known observation. We say an $R$-complex $M$ in $\dbcat R$ has \emph{finite projective dimension}, and write $\projdim_RM<\infty$, if it is quasi-isomorphic to a bounded complex of finite free $R$-modules; see \cite{Avramov/Foxby:1991a} for a discussion on homological dimensions for complexes.
 
 \begin{lemma}
     \label{le:eta-iso}
    Fix $M$ in $\dbcat R$. The following conditions are equivalent:
    \begin{enumerate}[\quad\rm(1)]
        \item $\projdim_RM<\infty$;
        \item  $\eta(M)$ is an isomorphism;
        \item $\rank_k\ker\eta(M)<\infty$.
    \end{enumerate}
 \end{lemma}

 \begin{proof}
 It is easy to verify that when $M$ is quasi-isomorphic to a finite free complex, the natural map
     \[
     \rhom_R(k,R)\lotimes_R M\longrightarrow \rhom_R(k,M)
     \]
     is a quasi-isomorphism, and hence that $\eta(M)$ is an isomorphism; thus (1)$\Rightarrow$(2).

     (2)$\Rightarrow$(3) is a tautology.

     (3)$\Rightarrow$(1) Set $d=\depth R$; thus $\eee^d(R)\ne 0$. We can assume $\hh(M)\ne 0$ so $\eee(M)\ne 0$. Since  $\eee^i(M)=0$ for $i<\depth_RM$ and $\eta(M)$ respects degrees, it follows that
     \[
     \eee^d(R)\otimes \tee_i(M)\subseteq \ker\eta(M)\quad\text{for $i>d-\depth_RM$.}
     \]
     Thus when $\ker\eta(M)$ has finite rank, $\tee_i(M)=0$ for $i\gg 0$. Since $M$ is in $\dbcat R$ it follows that $\projdim_RM<\infty$; see, for instance, \cite[Proposition~5.3.P.]{Avramov/Foxby:1991a}.
 \end{proof}
   
\subsection*{Stable cohomology}
We recall some facts about stable cohomology modules required in the sequel; see \cite{Avramov/Iyengar:2013s, Avramov/Veliche:2007}. Fix $M$ in $\dbcat R$. We write $\widehat{\eee}(M)$ for $\widehat{\ext}_R(k,M)$, the stable cohomology of the pair $(k,M)$.  By construction, there is a natural map $\iota(M)\colon \eee(M)\to \widehat{\eee}(M)$ and this fits into a long exact sequence 
\begin{equation}
\label{eq:scfunctor}
\to \eee(R)\otimes_k \tee(M)\xrightarrow{\ \eta(M)\ } \eee(M)\xrightarrow{\ \iota(M)\ } \widehat{\eee}(M)\to \Sigma (\eee(R)\otimes_k \tee(M))\to 
\end{equation}
of  graded $k$-vectorspaces where $\eta(M)$ is the map in \eqref{eq:eta-defn}. In the sequel, we exploit another interpretation of $\widehat{\eee}(M)$. This involves the maps $s_nM$ from \eqref{eq:smap}.

\begin{lemma}
\label{le:scfunctor}
The map $\widehat{\eee}(s_nM)\colon \widehat{\eee}(M)\to \widehat{\eee}(\Sigma^n\Omega^nM)$ is an isomorphism for each integer $n$. 
\end{lemma}

\begin{proof}
With $F$ the minimal free resolution of $M$, the map $s_nM$ is represented by the quotient map $F\to F_{\geqslant n}$. Consider the exact sequence
\[
0\longrightarrow F_{<n}\longrightarrow F\xrightarrow{\ s_nM\ } F_{\geqslant n}\longrightarrow 0\,.
\]
One has $\widehat{\eee}(F_{<n})=0$, by the construction of stable cohomology modules,   so applying $\widehat{\eee}(-)$ to the exact sequence above yields the desired result.
\end{proof}

The maps $s_nM$ from \eqref{eq:smap} give rise to maps
\[
M\longrightarrow \Sigma^n \Omega^n M \longrightarrow \Sigma^{n+1}\Omega^{n+1} M \longrightarrow\dots 
\]
in $\dcat(R)$. In cohomology, these induce maps
\[
\begin{tikzcd}
\eee(M) \arrow[d] \arrow[r] & \eee(\Sigma^n\Omega^n M) \arrow[d] \arrow[r] & \eee(\Sigma^{n+1}\Omega^{n+1} M) \arrow[d] \arrow[r] &\cdots  \\
\widehat{\eee}(M) \arrow[r, ,"\cong"] & \widehat{\eee}(\Sigma^n\Omega^n M) \arrow[r,"\cong"] & \widehat{\eee}(\Sigma^{n+1}\Omega^{n+1} M) \arrow[r,"\cong"] &\cdots  
\end{tikzcd}
\]
The isomorphisms are by the preceding lemma. The squares commute because of the naturality of the transformation $\eee(-)\to \widehat{\eee}(-)$. Thus $\iota(M)$ factors as
\[
\eee(M)\longrightarrow \mathrm{colim}_n \eee(\Sigma^n\Omega^n M) \longrightarrow \widehat{\eee}(M)\,.
\]
The result below is a variant of Mislin's~\cite{Mislin:1994a} description of stable cohomology.

\begin{lemma}
\label{le:Mislin}
    The map  $\mathrm{colim}_n \eee(\Sigma^n\Omega^n M) \to \widehat{\eee}(M)$ is an isomorphism.
\end{lemma}

\begin{proof}
In what follows we rely on \cite[Section 2]{Mislin:1994a} and \cite[Section 4]{Kropholler:1994a}.

For each integer $i$ set $\mathrm{F}^i(M)= \mathrm{colim}_n \eee^i(\Sigma^n\Omega^n M)$. Then the families 
\begin{align*}
\eee(M)&=\{\eee^i(M)\mid i\in\bbZ\}\,, \\
\widehat{\eee}(M)&=\{\widehat{\eee}{}^i(M)\mid i\in\bbZ\}\,, \\
\mathrm{F}(M)&=\{\mathrm{F}^i(M)\mid i\in\bbZ\}\,, 
\end{align*}
are cohomological functors from $\dbcat R$ to $k$-vectorspaces. By construction, there are natural transformation $\eee \to \mathrm{F}\to \widehat{\eee}$  of cohomological functors. We claim
\begin{enumerate}[\quad\rm(1)]
    \item when $M$ is a projective module $\mathrm{F}(M)=0$;
    \item if $\mathrm{G}$ is any cohomological functor that vanishes on projective modules, then any natural transformation $\mathrm{E}\to \mathrm{G}$ factors uniquely through $\mathrm{F}$.
\end{enumerate}
These properties of  $\mathrm{F}(M)$ mean that it is the $P$-completion of $\eee(M)$, in the sense of \cite[Section 2]{Mislin:1994a}, which is called the Mislin completion in \cite[Section 4]{Kropholler:1994a}; this gives the desired result because $P$-completions, when they exist, are unique and the $P$-completion of $\eee(M)$ is $\widehat{\eee}(M)$; see \cite[Section 4]{Kropholler:1994a}.

It thus remains to verify the properties of $\mathrm{F}$ stated above. When $M$ is projective so are its syzygy modules and hence the exact sequences defining the syzygies
\[
0\longrightarrow \Omega^{n+1}M\longrightarrow F_{n+1}\longrightarrow \Omega^{n}M\longrightarrow 0
\]
 are split-exact. Hence the induced map $\eee(\Sigma^{n}\Omega^{n}M)\to \eee(\Sigma^{n+1}\Omega^{n+1}M)$ is zero for each $n$.  It follows that $\mathrm{F}(M)=0$, justifying (1).  As to (2), the  construction of the natural transformation $\mathrm{F}\to \widehat{\eee}$ above only used the fact that $\mathrm{\eee}$ is a cohomological functor and that it vanishes on projective modules. Thus, the same argument yields that any natural transformation $\eee\to \mathrm{G}$ factors uniquely through $\mathrm{F}$.
\end{proof}

\section{Unstable elements in cohomology}
\label{se:Uspace}
As in the previous section, let $(R,\mathfrak{m},k)$ be a commutative noetherian local ring with maximal ideal $\fm$ and residue field $k$, and fix an $R$-complex $M$ in $\dbcat R$. Recall the map $\iota(M)$ from the last section; see \eqref{eq:etfunctors} and \eqref{eq:scfunctor}. The focus of this work is on the functor that assigns $M$ to the graded $k$-vectorspace
\begin{equation}
\label{eq:U-defn}
U(M)\coloneqq \ker(\iota(M)\colon \eee(M) \longrightarrow  \widehat{\eee}(M))\,.    
\end{equation}
where the grading is inherited from $\eee(M)$. We think of $\eee(M)$ as the \emph{cohomology} of $M$ and $\widehat{\eee}(M)$ as its \emph{stable cohomology}. Thus $U(M)$ is the subspace of \emph{unstable} elements in $\eee(M)$, whence the title of this work. We are particularly interested to know when $U(M)=0$ holds; this is relevant to the question regarding Bass series discussed in the Introduction; see Corollary~\ref{co:summands} and Section~\ref{se:BassSeries}.  The subspace $U(M)$ is a covariant version of the invariant $\xi(M)$, defined to be the kernel of the map $\ext_R(M,k)\to \widehat{\ext}_R(M,k)$, investigated by Martsinkovsky~\cite{Martsinkovsky:1996b, Martsinkovsky:2000a}.

The subspace $U(M)$ is interesting only when $R$ is \emph{singular}, that is to say, not regular, because of the following result. In view of Lemma~\ref{le:eta-iso}, the first part is just a reformulation of the characterization of regular rings due to Auslander and Buchsbaum, and Serre; see~\cite[Theorem~2.2.7]{Bruns/Herzog:1998a}. The second part is a reformulation, using Lemma~\ref{le:AW} below, of a result of Lescot~\cite[1.6]{Lescot:1983}; see also \cite[Theorem~5.1.8]{Avramov/Veliche:2007}.

\begin{theorem}[\cite{Lescot:1983}]
\label{th:U-Rk}
The ring $R$ is regular and only if  $U(M) = \eee(M)$ for each $M$ in $\dbcat R$. When  $R$ is singular, $U(k)=0$. \qed
\end{theorem}

We consider a filtration of $U(M)$ that is motivated by Lemma~\ref{le:Mislin}. For each integer $n$ set
\begin{equation}
\label{eq:Ufiltration}
U_n(M) \coloneqq \ker\left(\eee(s_nM)\colon \eee(M)\longrightarrow \eee(\Sigma^n\syz^n M)\right)\,.    
\end{equation}
This is a graded $k$-vector subspace of $U(M)$. Lemma~\ref{le:Mislin} implies that one gets an increasing, exhaustive, filtration on $U(M)$, in that
\[
\{0\}= U_{i}(M) \subseteq \cdots \subseteq \bigcup_n U_{n}(M) = U(M)\,.
\]
where $i=\inf\hh_*(M)$; one has $U_i(M)=\{0\}$ because $\Sigma^i\Omega^iM\simeq M$.  In fact, more is true, and to explain this we invoke another interpretation of $U(M)$ which plays a key role in the sequel.  The map $\eta(M)$ is the one from \eqref{eq:eta-defn}; we use the same notation also for its restriction to subspaces.

\begin{lemma}
\label{le:Un-description}
Fix $M$ in $\dbcat R$. One has $U(M) = \image \eta(M)$, so $U(M)=0$ if and only if $\eta(M)=0$. Moreover, for each integer $n$ there is an exact sequence 
\[
\eee(R)\otimes_k \tee_{<n}(M)\xrightarrow{\ \eta(M)\ } \eee(M) \longrightarrow  \eee(\Sigma^n\Omega^nM) \longrightarrow
            \Sigma \eee(R)\otimes_k \tee_{<n}(M)
\]
Hence $U_n(M)= \image\big(\eee(R)\otimes_k \tee_{<n}(M)\xrightarrow{\eta(M)} \eee(M)\big)\,.$
\end{lemma}

\begin{proof}
Given the exact sequence \eqref{eq:scfunctor}, the first part of the result is clear from the definition of $U(M)$. We have only to justify the exactness of the displayed sequence. Let $F$ be a minimal resolution of $M$; thus $\Sigma^n\Omega^nM=F_{\geqslant n}$. Consider the exact sequence of complexes
 \[
 0\longrightarrow F_{<n}\longrightarrow F\xrightarrow{\ s_nM\ } F_{\geqslant n}\longrightarrow 0\,.
 \]
  Applying $\eee(-)$ yields that the lower row in the diagram below is exact
 \[
 \begin{tikzcd}
 \eee(R)\otimes_k\tee(F_{<n}) \arrow[d,"\eta(F_{<n})" swap, "\cong"]\arrow[r,hookrightarrow]
        &  \eee(R)\otimes_k\tee(M)\arrow[d,"\eta(M)"]  \\
\eee(F_{<n}) \arrow[r] & \eee(M) \arrow[r, " \eee(s_nM)"] & \eee(\Sigma^n\Omega^nM) \longrightarrow \Sigma \eee(F_{<n})
 \end{tikzcd}
 \]
The square is commutative by the functoriality of $\eee(-)$ and $\eta(-)$. The map $\eta(F_{<n})$ is an isomorphism because  $F_{<n}$ is a finite free complex; see Lemma~\ref{le:eta-iso}. It is clear that $\tee(F_{<n})$ identifies with $\tee_{<n}(M)$. This justifies the exactness of the sequence.
\end{proof}

Here is another description of $U_n(M)$.

\begin{lemma}
\label{le:Un-ses}
For $M$ in $\dbcat R$ and integer $n$ there is  an exact sequence
\[
0\longrightarrow U_n(M)\longrightarrow  U(M)\xrightarrow{\ U(s_nM)\ } U(\Sigma^n\Omega^nM)\longrightarrow 0\,.
\]
Hence $U_n(M)=U(M)$ if and only if $\eee(\syz^nM)\to\widehat{\eee}(\syz^nM)$ is one-to-one.
\end{lemma}

\begin{proof}
Consider the diagram below, which is commutative because of the naturality of the constructions involved:
    \[
\begin{tikzcd}[column sep=large]
    \eee(R) \otimes_k \tee(M) 
        \arrow[d, twoheadrightarrow, "1\otimes \tee(s_nM)"] \arrow[r,"\eta(M)"] 
                & \eee(M)  \arrow[r,"\iota(M)"] \arrow[d,"\eee(s_nM)" ] & \widehat{\eee}(M) \arrow[d,"\widehat{\eee}(s_nM)", "\cong" swap]\\
   \phantom{testin}\eee(R) \otimes_k  \tee(\Sigma^n\Omega^nM) \arrow[r,"\eta(\Sigma^n\Omega^nM)"]  
                & \arrow[r,"\iota(\Sigma^n\Omega^n M)"]  \eee(\Sigma^n\Omega^nM)     & \widehat{\eee}_R(\Sigma^n\Omega^nM)
\end{tikzcd}
\]
The map $\tee(s_nM)$ is surjective, by Lemma~\ref{le:tee-syz},  the isomorphism on the right holds by Lemma~\ref{le:scfunctor}, and the rows are exact by \eqref{eq:scfunctor}. The commutativity of the square on the left implies that the map $U(s_nM)\colon U(M)\to U(\Sigma^n\Omega^nM)$ is surjective. Moreover, since $\tee_i(s_nM)$ is an isomorphism for $i\ge n$, the kernel of $U(s_nM)$ is the image of $\eee(R)\otimes \tee_{<n}(M)$ under $\eta(M)$, that is to say, $U_n(M)$; see Lemma~\ref{le:Un-description}.
\end{proof}

\subsection*{Annihilators}
To track the behavior of $U(M)$ under change of modules it is also helpful to consider the graded $k$-vectorspaces
\begin{gather*}
    A_n(M)\coloneqq \{\alpha\in \eee(R)\mid  \eta(M)(\alpha\otimes-) = 0 \text{ on } \tee_{<n}(M)\} \\
    A(M) \coloneqq \bigcap_{n\in\bbZ} A_n(M)\,.
\end{gather*}
The subspaces $\{A_n(M)\}_n$ form a descending filtration on  $\eee(R)$, with $A_n(M)=\eee(R)$ for all $n\le \inf \hh_*(M)$;
see Lemma~\ref{le:inf}. The assignment $M\mapsto A(M)$ defines a covariant functor on $\dbcat R$. Here are some obvious properties of this functor; there are analogues also for the functors $M\mapsto A_n(M)$, but we do not have use for them in this work.

\begin{lemma}
\label{le:A-basics}
For $R$-complexes $M,N$ in $\dbcat R$ the following statements hold:
\begin{enumerate}[\quad\rm(1)]
    \item
    $A(M)=\eee(R)$ if and only if $U(M)=0$. 
    \item 
    $A(M)=A(\Sigma^nM)$  for any integer $n$; 
    \item
    $A(M)\subseteq A(\Omega^nM)$  and any integer $n$;
    \item
    $A(M\oplus N)=A(M)\cap A(N)$; 
    \item 
     $A(M)\subseteq A(N)$ if there exists a morphism $f\colon M\to N$ in $\dbcat R$ such that $\tee(f)$ is onto.
\end{enumerate}     
\end{lemma}

\begin{proof}
Parts (1), (2) and (4) are straightforward to verify, and (3) is a special case of (5), given Lemma~\ref{le:tee-syz}. Part (5) follows from the commutative diagram
        \[
\begin{tikzcd}[column sep=large]
    \eee(R) \otimes_k \tee(M) 
        \arrow[d, twoheadrightarrow, "1\otimes \tee(f)" swap] \arrow[r,"\eta(M)"] 
                & \eee(M) \arrow[d,"\eee(f)"] \\
    \eee(R) \otimes_k  \tee(N) \arrow[r,"\eta(N)"swap]  
                & \eee(N) 
\end{tikzcd}
\]
where the map on the left is surjective by hypothesis.
\end{proof}

The condition $A(M)=0$ is also of interest; see, for instance, Theorem~\ref{th:ps} below. When $\eta(M)$ is one-to-one, for example, when it is an isomorphism, $A(M)=0$, but the converse does not hold.

\begin{example}
\label{ex:A0}
Fix $M$ in $\dbcat R$ with $\injdim_RM$ finite but $\projdim_RM$ infinite; this can happen only if $R$ is not Gorenstein. Thus $A(M)=0$, by Corollary~\ref{co:A-fin-hom-dim} below.  On the other hand, since $\projdim_RM$ is infinite, $\ker \eta(M)$ is infinite dimensional; see Lemma~\ref{le:eta-iso}.

For instance, take  $M=K\otimes_RI$ where $K$ is the Koszul complex on some finite generating set for the maximal ideal of a non-Gorenstein ring $R$ and $I$ is the injective hull of the residue field of $k$. 
\end{example}

In the next result, $(-)^\dagger$ denotes the local duality functor for rings with dualizing complexes; see Section~\ref{se:prelims}.
 
\begin{lemma}
\label{le:A-duality}
When $R$ has a dualizing complex $A(M)=A(M^\dagger)$ for any $R$-complex $M$ in $\dbcat R$.
\end{lemma}

\begin{proof}
The vectorspace $A(M)$ can also be understood as the kernel of the map
\[
\eee(R)\longrightarrow \Hom_k(\tee(M),\eee(M))\,.
\]
that is adjoint to the map $\eta(M)$. This map is the cohomological shadow of the first of the following maps in $\dcat(R)$:
\begin{align*}
\rhom_R(k,R)
  &\longrightarrow    \rhom_k(k\lotimes_RM,\rhom_R(k,M)) \\
  &\xrightarrow{\ \simeq\ } \rhom_R(k\lotimes_RM,M) \\
  &\xrightarrow{\ \simeq\ } \rhom_R(k,\rhom_R(M,M)) 
\end{align*}
 The isomorphisms are adjunctions. The composition is obtained by applying $\eee(-)$ to the homothety map  $h_M\colon R\to \rhom_R(M,M)$. Thus we deduce that
\[
A(M) = \ker (\eee(R) \xrightarrow{\ \eee(h_M)\ } \eee(\rhom_R(M,M))\,.
\]
For any $R$-complex $X$ it is straightforward to verify that the functor $\Hom_R(-,X)$ commutes with the homothety maps, in that the following diagram is commutative
\[
\begin{tikzcd}[row sep=small]
    & \Hom_R(M,M) \arrow[dd] \\
R\arrow[ur] \arrow[dr]  & \\
    & \Hom_R(\Hom_R(M,X),\Hom_R(M,X))
\end{tikzcd}
\]
In the derived category, and with $X=\omega_R$, this yields a commutative diagram 
\[
\begin{tikzcd}[row sep=small]
    & \rhom_R(M,M) \arrow[dd,"\simeq"] \\
R\arrow[ur,"h_M"] \arrow[dr,"h_{M^\dagger}" swap]  & \\
    & \rhom_R(M^\dagger,M^\dagger)
\end{tikzcd}
\]
The isomorphism is by local duality, and holds because $M$ is in $\dbcat R$; see \cite[\href{https://stacks.math.columbia.edu/tag/0A81}{Section 0A81}]{stacks-project}. Applying $\eee(-)$ to this diagram gives the commutative diagram:
\[
\begin{tikzcd}[row sep=small]
    & \eee(\rhom_R(M,M)) \arrow[dd,"\cong"] \\
\eee(R)\arrow[ur,"\eee(h_M)"] \arrow[dr,"\eee(h_{M^\dagger})" swap]  & \\
    & \eee(\rhom_R(M^\dagger,M^\dagger))
\end{tikzcd}
\]
It follows that $\ker \eee(h_M)=\ker \eee(h_{M^\dagger})$, which is as desired.
\end{proof}

Here is a corollary of the previous result; the new part concerns the case when the injective dimension of $M$ is finite.

\begin{corollary}
\label{co:A-fin-hom-dim} 
For any local ring $R$ and $M\in \dbcat R$, one has $A(M)=0$ when either $\projdim_RM$ or $\injdim_RM$ is finite.
\end{corollary}

\begin{proof}
When $\projdim_RM$ is finite, $\eta(M)$ is an isomorphism and $A(M)=0$. 

Suppose $\injdim_RM$ is finite. Let $R\to \widehat{R}$ denote the completion of $R$ at its maximal ideal, $\fm$, and set $\widehat{M}=\widehat{R}\otimes_RM$. Since $\ext_R(k,M)$ and $\tor^R(k,M)$ are $\fm$-torsion, and $\widehat{R}$ is flat as an $R$-module, one has natural isomorphisms 
\[
\ext_R(k,M)\cong \ext_{\widehat R}(k,\widehat{M})\quad\text{and}\quad
\tor^R(k,M)\cong \tor^{\widehat R}(k,\widehat{M})\,.
\]
In particular $\injdim_{\widehat{R}}(\widehat{M})=\injdim_RM <\infty$ is finite. Thus we can replace $R$ and $M$ by $\widehat{R}$ and $\widehat{M}$ and assume $R$ is complete, and hence that it admits a dualizing complex. Since $k^\dagger \simeq k$,
local duality yields an isomorphism
\[
\ext_R(M^\dagger,k)\cong \ext_R(k,M)\,,
\]
and this gives the equality below:
\[
\projdim_RM^{\dagger}=\injdim_RM <\infty\,.
\]
The injective dimension of $M$ is finite, by hypothesis. Thus $A(M)=A(M^\dagger)=0$, where the first equality is by Lemma~\ref{le:A-duality}.
\end{proof}

Here is a consequence of Corollary~\ref{co:A-fin-hom-dim}. It recovers \cite[1.7,1.8]{Lescot:1986}, which already contains \cite[Theorem III]{Ghosh/Gupta/Puthenpurakal:2018} that deals only with the special case $M=k$.

\begin{corollary}
\label{co:summands}
If $U(M)=0$, then for  $n\ge \inf\hh_*(M)$ the projective dimension and the injective dimension of any nonzero direct summand of $\Omega^nM$ is infinite.
\end{corollary}

\begin{proof}
When $U(M)=0$ one has $A(M)=\eee(R)\ne 0$, and hence $A(\Omega^nM)\ne 0$ for each $n\ge \inf\hh_*(M)$, by Lemma~\ref{le:A-basics}. By the same token $A(N)\ne 0$ for any nonzero direct summand of $\Omega^n(M)$. It remains to recall Corollary~\ref{co:A-fin-hom-dim}.
\end{proof}

Here is an aspect of $A(-)$ that is not immediately obvious; the proof uses certain multiplicative structures on the functors involved. These begin to play a bigger role in Section~\ref{se:products}.

\begin{proposition}
    \label{pr:A-change-of-rings}
Let $\varphi\colon R\to S$ be a finite map and fix $N$ in $\dbcat S$. Viewing $S$ and $N$ as $R$-complexes by restriction of scalars along $\varphi$, one has that
\[
A(N)\supseteq A(S) = \ker(\eee(R)\xrightarrow{\ \ext_R^*(k,\varphi)\ } \eee(S))\,.
\]
\end{proposition}

\begin{proof}
Since $S$ and $k$ are $R$-algebras, $k\otimes_RS$ is a $k$-algebra, acting on $\Hom_R(k,N)$ for any $S$-module $N$ as follows: given $x\in (k\otimes_RS)$ and an $R$ linear map $f\colon k\to N$, the map $f\cdot x$ is the composition of maps
\[
k\xrightarrow{\ 1\mapsto x} (k\otimes_RS)\xrightarrow{\ f\otimes S} (N\otimes_RS) \longrightarrow N
\]
where the map on the right is multiplication. Moreover the natural map
\[
\Hom_R(k,R)\otimes_k (k\otimes_RN) \longrightarrow \Hom_R(k,N)
\]
is compatible with these actions. 

The (derived versions of these maps) induce a natural action of the $k$-algebra $\tee(S)$ on $\tee(N)$ and $\eee(N)$, for  $N$ in $\dbcat S$,  and the map
 \[
 \eee(R)\otimes_k\tee(N) \xrightarrow{\ \eta(N)\ } \eee(N)
 \]
is compatible with these actions. It follows that $A(N)\supseteq A(S)$. 

Moreover, for $N=S$, the left-hand-side $\eta(S)$ is generated as an $\tee(S)$-module by $\eee(R)\otimes_k 1$. Hence the $\tee(S)$-linearity of $\eta(S)$ implies the equality on the left:
\[
A(S)=\{\alpha\in \eee(R)\mid \eta(\alpha\otimes 1)=0\}=\ker \eee(\varphi)\,.
\]
The equality on the right is clear. This  is the desired result.
 \end{proof}

Proposition~\ref{pr:A-change-of-rings} and  Corollary~\ref{co:A-fin-hom-dim} yield another proof of~\cite[Theorem~5.5]{Peskine/Szpiro:1973} by Peskine and Szpiro. The argument mimics  Lescot's proof of ~\cite[Theorem~1.9]{Lescot:1986}; it is beautiful and bears repeating.

\begin{theorem}
\label{th:ps}
If there exists an ideal $I\subset R$ such that $\injdim_R(R/I)$ is finite, then $R$ is Gorenstein.
\end{theorem}

\begin{proof}
Given such an $I$, set $S=R/I$, let $R\to S$ be the canonical surjection, and consider the exact sequence of graded $k$-vectorspaces
\[
0\longrightarrow A(S)\longrightarrow \eee(R)\longrightarrow \eee(S)
\]
given by Proposition~\ref{pr:A-change-of-rings}. Since $\injdim_RS$ is finite, $\rank_k\eee(S)$ is finite and also 
$A(S)=0$; the latter conclusion is by Corollary~\ref{co:A-fin-hom-dim}. Thus $\rank_k\eee(R)$ is finite; equivalently, $\injdim_RR$ is finite, so $R$ is Gorenstein.
\end{proof}

\subsection*{Revisiting the work of Lescot}
The map $\eta(M)$ from \eqref{eq:eta-defn} is closely related to the map studied by Lescot in~\cite{Lescot:1986}. This is explained in the following paragraphs.  We take this opportunity to  present some of Lescot's work from a newer perspective, and notation, partly to pave the way for the  material presented in later sections. 

\subsection*{A homology product}
In the remainder of this section we assume that the local ring $R$ has a dualizing complex, $\omega_R$, and write $(-)^\dagger$ for the corresponding local duality functor; see Section~\ref{se:prelims}. For each $M$ in  $\dbcat R$ the evaluation map 
\[
M^\dagger \lotimes_R M \longrightarrow \omega_R
\]
induces the map of graded $k$-vectorspaces on the right:
\[
    \tee(M^\dagger)\otimes_k \tee(M)
        \longrightarrow \tee(M^\dagger \lotimes_R M)
            \longrightarrow  \tee(\omega_R)\,.
\]
The map on the left is the K\"unneth map. Composing them gives the map of graded $k$-vectorspaces
\begin{equation}
 \label{eq:theta-defn}    
\theta(M)  \colon \tee(M^\dagger )\otimes_k \tee(M) \longrightarrow \tee(\omega_R)\,.
\end{equation}
This map coincides with the one introduced by Lescot in \cite{Lescot:1986}. 

Indeed, writing $\rgam_{\fm}(-)$ for the local cohomology functor with support at $\fm$, and $E$ the injective hull of the $R$-module $k$, for $M\in\dbcat R$ one has an isomorphism
\[
\tee(M)\cong \tee(\rgam_\fm(M))\,.
\]
Moreover $\rgam_{\fm}(M^\dagger) \simeq \Hom_R(M,E)$ and in particular $\rgam_{\fm}(\omega_R)\simeq E$. Hence the map $\theta(M)$ defined above coincides with the map
\[
\tee(\Hom_R(M,E)) \otimes_k \tee(M) \to \tee(E)
\]
induced by the evaluation map $\Hom_R(M,E) \otimes_R M\to E$. Thus \eqref{eq:theta-defn} is the homology product from \cite[1.1]{Lescot:1986}. Given this, and following Lescot, for each integer $n$ we set
\begin{align*}
W^n(M)&\coloneqq \image\big(\theta(M)\colon  \tee(M^\dagger)\otimes_k \tee_{<n}(M) \to \tee(\omega_R)\big)  \\
F^n(M)& \coloneqq \{\tau\in\tee(M^\dagger)\mid \theta(M)(\tau\otimes -)=0 \text{ on } \tee_{<n}(M) \}\,.
\end{align*}
Moreover set
\begin{align*}
W(M) &\coloneqq \bigcup_{n\in\bbZ} W^n(M) = \image \theta(M)\,, \\
F(M) &\coloneqq \bigcap_{n\in\bbZ} F^n(M) = \{\tau\in\tee(M^\dagger)\mid \theta(M)(\tau\otimes -)=0 \text{ on } \tee(M) \}\,.
\end{align*}
Here is the relationship between these subspaces and ones introduced earlier.

\begin{lemma}
\label{le:AW}
The maps $\theta(M)$ and $\eta(M)$ are adjoint, and hence $U(M)=0$ if and only if $W(M)=0$; moreover, for each integer $n$ one has 
\[
A_n(M)^\perp = W^n(M) \quad\text{and}\quad U_n(M)^\perp = F^n(M)
\]
with respect to the canonical pairings.
\end{lemma}

\begin{proof}
    An adjoint of the map $\theta(M)$ is the map 
    \[
    \Hom_k(\tee(\omega_R),k)\otimes_k \tee(M)\longrightarrow \Hom_k(\tee(M^\dagger),k)\,.
    \]
    This identifies with the map $\eta(M)$, once we take into account the isomorphism (for $M$ and for $M=R$) of graded $k$-vectorspaces
    \[
    \Hom_k(\tor^R(k,M^\dagger),k) \cong \ext_R(k,M)
    \]
    given by local duality; see~\cite[Section 47.18]{StacksProject}. This justifies the claim that $\theta(M)$ and $\eta(M)$ are adjoint to each other. The remaining claims are a straightforward consequence; see the discussion on adjoint maps further below.
\end{proof}

Lemma~\ref{le:AW} allows one to translate properties of the functors $U(-)$ and $A(-)$ into statements about $W(-)$ and $F(-)$. Here is a sample. 

\begin{itemize}
    \item Proposition~\ref{pr:A-change-of-rings} translates to  (1.4) and (1.5) in \cite{Lescot:1986}.
    \item Lemma~\ref{le:A-duality} translates to the statement that $W(M)=W(M^\dagger)$.
    \item Corollary~\ref{co:A-fin-hom-dim} translates to \cite[1.7, 1.8]{Lescot:1986}.
    \end{itemize}

\subsection*{Adjoint of a pairing}
Let $k$ be a field, and let $A,B,C$ be graded $k$-vectorspaces equipped with  $k$-linear map
\[
\theta \colon A\otimes_k B\longrightarrow C\,.
\]
We assume that the $k$-vectorspaces $A_i$, $B_i$ and $C_i$ are finite dimensional for $i$ and equal to $0$ for $i\ll 0$. The relevant example for us is the map $\theta(M)$ from \eqref{eq:theta-defn}. With $(-)^\vee$ denoting the graded $k$-vectorspace dual,  consider the adjoint map
\begin{gather*}
\eta  \colon C^\vee \otimes_k B\longrightarrow A^\vee \,, \\
\eta(f\otimes b) \mapsto \left[a\mapsto (-1)^{|a||b|}f(\theta(a\otimes b)) \right]
\end{gather*}
As the notation suggests, this corresponds to the map $\eta(M)$ from \eqref{eq:eta-defn}. The following assertions can be verified directly:
\begin{enumerate}[\quad\rm(1)]
    \item 
    $\theta = 0$ if and only if $\eta = 0$;
    \item
There is an equality
\[
\image(A\otimes_k B\xrightarrow{\theta} C) = \{f\in C^\vee\mid \eta(f\otimes-)=0 \text{ on } B\}^\perp \,.
\]
The orthogonal subspace is with respect to the canonical map $C^\vee\otimes_k C \to k$.
\item 
There is an equality
\[
\image(C^\vee\otimes_k B\xrightarrow{\eta} A^\vee) = \{a\in A\mid \theta(a\otimes-)=0 \text{ on } B\}^\perp \,.
\]
The orthogonal subspace is with respect to the canonical map $A^\vee\otimes_k A \to k$.
\end{enumerate}

\section{The Lescot invariant}
\label{se:sigma}
Let $(R,\fm,k)$ be a noetherian local ring. In view of Theorem~\ref{th:bass-series-formula}, it is natural to consider the following invariant of $R$:
\begin{equation}
\label{eq:sigma-defn}    
\sigma(R)\coloneqq\inf\{n\ge 0\mid U(\syz^nM)=0
        \text{~for all $M \in \rmod R$}\}.
\end{equation}  
If there exists no such $n$, then $\sigma(R)=\infty$. Lemma~\ref{le:AW} allows an alternative interpretation of this invariant:
\[
\sigma(R)=\inf\{n\ge 0\mid \lescot(\syz^nM)=0\text{ for all $M\in\rmod R$}\} \,.   
\]
Thus $\sigma(R)$ is precisely the invariant introduced by Lescot~\cite{Lescot:1986}.  From Lemmas~\ref{le:Un-description} and the description of $A(M)$ we get that
    \[
    U(\syz^nM)=0\text{ if and only if } U_n(M)=U(M)\,, 
        \text{ if and only if } A_n(M)=A(M)\,.
    \]
    Hence $\sigma(R)$ can also be calculated as 
    \begin{equation}
    \label{eq:sigma-equivalent}
    \begin{aligned}
    \sigma(R) 
    & =\inf\{n\geq 0\mid U_{n}(M)=U(M) \text{ for all $M\in\rmod R$} \} \\
    & =\inf\{n\geq 0 \mid {F}^{n}(M)=F(M) \text{ for all $M\in\rmod R$} \} 
    \end{aligned}
    \end{equation}

Since $U(R)\neq 0$ one always has that $\sigma(R)\geq 1$. Lescot~\cite{Lescot:1986} asks:
\begin{question}
\label{qu:lescot}
    Is $\sigma(R)<\infty$ for any local ring $R$?
\end{question}
The rest of this work is concerned with this question. For a start, we do not even have an answer for the weaker question:

\begin{question}
\label{qu:lescotM}
For a finitely generated $R$-module $M$, is $U(\syz^nM)=0$ for $n\gg 0$?
\end{question}

Evidently, this holds if $M$ has finite projective dimension. This is all we know for specific $M$ over a general local ring. 
We fare better with Question~\ref{qu:lescot}. Here is a first observation.

 \begin{lemma}
 \label{le:sigma-depth}
     One  has $\sigma(R)\ge \depth R+1$.
 \end{lemma}

 \begin{proof}
     Indeed, set $d= \depth R$ and $M=R/(\boldsymbol{x})$, where $\boldsymbol{x}$ is a maximal regular sequence in $R$. 
Since $\projdim M=d$, one has $\syz^dM\cong R^b$ for some nonzero integer $b$, hence $U(\syz^dM)\supseteq U(R)\neq 0$.
\end{proof}

The result below recovers \cite[Proposition 3.3]{Lescot:1986}; see the proof of Corollary~\ref{co:gorenstein} for another perspective on it.

\begin{proposition}
    \label{pr:sigma-gor}
When the local ring $R$ is Gorenstein, $\sigma(R)=\dim R+1$.    
\end{proposition}

\begin{proof}
Fix a nonzero finitely generated $R$-module $M$. One has $\eee(R)=\eee^d(R)$ for $d=\dim R$ as $R$ is Gorenstein.  Since $\eee^i(M)=0$ for $i<0$ the map $\eta(M)$ is zero when restricted to the subspace $\eee(R)\otimes_k \tee_{\geqslant d+1}(M)$,  by degree considerations. Thus Lemma~\ref{le:Un-description} implies $U_{d+1}(M)=U(M)$ and then Lemma~\ref{le:Un-ses} yields $U(\syz^{d+1}M)=0$. Hence $\sigma(R)\le d+1$. The reverse inequality holds by Lemma~\ref{le:sigma-depth}.
\end{proof}

Lescot~\cite[Proposition 3.4]{Lescot:1986} also proves that $\sigma(R)$ is finite when $R$ is Golod; we revisit this result in the next section; see Corollary~\ref{co:golod}. Next we record a result tracking the Lescot invariant $\sigma(-)$ under change of rings. Its proof uses the observation below.

\begin{lemma}
  \label{le:lescot-complex} 
For any $M$ in $\dbcat R$ one has
\[
    U(\syz^nM)=0 \quad\text{for $n\ge \sigma(R)+\sup\hh_*(M)$.}
\]
\end{lemma}

\begin{proof}
We can assume $\sigma(R)$ is finite. Set $s=\sup\hh_*(M)$. It is easy to verify that for each $i\ge 0$ one has a quasi-isomorphism $\syz^{i+s}M\cong \syz^i(\syz^sM)$. By the choice of the integer $s$, the natural surjection $\syz^sM\to \hh_0(\syz^sM)$ is a quasi-isomorphism. Hence for $i\ge \sigma(R)$ one gets
\[
U(\syz^{i+s}M)= U(\syz^i(\syz^sM)) = U(\syz^i(\hh_0(\syz^sM))) = 0\,.
\]
This justifies the claim.
\end{proof}

\subsection*{Gorenstein maps}
A finite local homomorphism $\varphi\colon R\to S$ is \emph{Gorenstein} if $g\coloneqq \projdim_RS$ is finite and $\rhom_R(S,R)\simeq \Sigma^{-g} S$ in $\dbcat S$.

\begin{theorem}
\label{th:gor-map}
If $\varphi\colon R\to S$ is a finite Gorenstein map, then 
\[
\sigma(R)\le \sigma(S)+\projdim_RS\,.
\]
\end{theorem}

\begin{proof}
We use the interpretation of $\sigma(R)$ in terms of the vanishing of $W(-)$. To begin with we can assume $R$ and $S$ share a common residue field. 

Let $g\coloneqq \projdim_RS$. Since $g$ is finite, for any $R$-complex $M$ in $\dbcat R$, the $S$-complex  $S\lotimes_R M$ is in $\dbcat S$. One thus gets a functor 
\[
\varphi^*\coloneqq S\lotimes_R-\colon \dbcat R\longrightarrow \dbcat S\,.
\]
The following properties of this functor are needed in what follows:
\begin{enumerate}[\quad\rm(1)]
    \item
    $\varphi^*\omega_R \cong \Sigma^g \omega_S$;
    \item
    $\varphi^*(M^\dagger) \cong (\varphi^*M)^\dagger$ for each $M\in\dbcat R$;
    \item 
    $\varphi^*(\syz^nM)\cong \syz^{n}(\varphi^*M)$ for each $M\in \dbcat R$ and each $n$.
\end{enumerate}

Indeed, since $g$ is finite one gets the second isomorphism below:
\[
\omega_S\cong \rhom_R(S,\omega_R)\cong \rhom_R(S,R)\lotimes_R\omega_R \cong \Sigma^{-g} S\lotimes_R\omega_R\,.
\]
The first one is standard---see \cite[\href{https://stacks.math.columbia.edu/tag/0AX0}{Tag 0AX0}]{stacks-project}---and the third one holds because $\varphi$ is Gorenstein. This gives (1).

Given (1), the isomorphism in (2) is the composition of isomorphisms
\begin{align*}
S\lotimes_R \rhom_R(M,\omega_R) &\xrightarrow{\ \cong\ } 
        \rhom_S(S\lotimes_R M,S\lotimes_R\omega_R) \\
        &\xrightarrow{\ \cong\ } \rhom_S(S\lotimes_R M, \Sigma^g \omega_S)\,,       
\end{align*}
where the first one holds because $\projdim_RS$ is finite.

(3) If $F$ is a minimal free resolution of $M$, then $S\otimes_RF$ is a minimal free resolution of $S\lotimes_RM$.  The stated isomorphism is clear from the definition of syzygies. This does not use the Gorenstein property of $\varphi$.

The associativity isomorphism $k\lotimes_R M\cong k\lotimes_S (S\lotimes_RM)$ means that one has a natural isomorphism of graded $k$-vectorspaces
\[
\tee^{\varphi}(M)\colon \tee^R(M)\xrightarrow{\ \cong\ } \tee^S(\varphi^*M)\,.
\]
In particular, from (1) above one gets an isomorphism of graded $k$-vectorspaces
\[
\tee^{\varphi}(\omega_R) \colon \tee^R(\omega_R)\xrightarrow{\ \cong\ } 
    \tee^S(\Sigma^{g}\omega_S) \cong \Sigma^g \tee^S(\omega_S)\,.
\]

With these isomorphisms one gets a commutative diagram:
\[
\begin{tikzcd} 
    \tee^R(M)  \otimes_k  \tee^R(M^\dagger) \arrow[d, xshift=-25pt,"\tee^{\varphi}(M)", "\cong" swap] 
                \arrow[d, xshift=25pt,"\tee^{\varphi}(M^\dagger)", "\cong" swap] \arrow[r] 
                    & \tee^R(\omega_R)\arrow[d, "\tee^{\varphi}(\omega_R)", "\cong" swap] \\
    \tee^S(\varphi^*M) \otimes_k  \tee^S (\Sigma^g (\varphi^*M)^\dagger )  
            \arrow[r] 
                    & \tee^S(\Sigma^g \omega_S)       
\end{tikzcd}
\]
It follows from this that the natural map
\[
\lescot_R(M)\cong \lescot_S(\varphi^*M)
\]
is an isomorphism. 
    
We can assume $\sigma(S)$ is finite.  The isomorphism above yields the first one below:
\begin{align*}
    \lescot_R(\syz^{n}M) 
        &\cong \lescot_S(\varphi^*(\syz^nM)) \\
        &\cong \lescot_S(\syz^{n}(\varphi^*M)) \\
        & =0\,.
\end{align*}
The second one holds by property (3), whereas the equality holds because 
\[
\sup\hh_*(\varphi^*M)=\sup\tor^R_*(S,M) \le g\,.
\]
Since $M$ as arbitrary, we get that $\sigma(R)\le \sigma(S)+g$ as claimed.
\end{proof}

The following consequence of the preceding result recovers observation (b) on Page 288 of \cite{Lescot:1986}.

\begin{corollary}
\label{co:sigma-nzd}
    If $x\in R$ is not a zero-divisor, then $\sigma(R)\leq \sigma(R/Rx)+1$.
\end{corollary}

\begin{proof}
One has $\projdim_R(R/Rx)=1$ and $\rhom_R(R/Rx,R)\simeq \Sigma^{-1}R/Rx$ in $\dcat(R/Rx)$. Thus the preceding result applies and yields the desired inequality.
\end{proof}

\begin{remark}
It follows from Lemma~\ref{le:sigma-depth} and Corollary~\ref{co:sigma-nzd} that for any maximal regular sequence $\boldsymbol{x}$ for $R$ one has
\[
\depth R+1 \le \sigma(R)\le \sigma(R/R\boldsymbol{x}) + \depth R\,.
\]
In particular, a positive answer to Question~\ref{qu:lescot} for artinian rings implies a positive answer for any Cohen-Macaulay ring.
\end{remark}

 \section{Products in cohomology}
 \label{se:products}
 Let $(R,\fm,k)$ be a local ring. We consider 
 \[
 \eee(k)\coloneqq\ext_R(k,k)
 \]
 as a graded $k$-algebra with composition products; this coincides with the Yoneda product, up to a sign. We recall that $\eee(k)$ is the universal enveloping algebra of a graded Lie algebra~\cite[Section~10]{Avramov:1998a}. For any $M\in \dcat(R)$ the graded $k$-vectorspace $\eee(M)$ is a right module over $\eee(k)$ and the graded $k$-vectorspace $\tee(M)$ is a left module over $\eee(k)$. 

 This induces a right $\eee(k)$-module structure on $\eee(R)\otimes_k\tee(M)$. Since $\eee(k)$ is primitively generated, it suffices to describe the action of such elements: For any primitive element $\zeta$ in $\eee(k)$, and elements $\alpha\in\eee(R)$ and $x\in\tee(M)$ one has
 \begin{equation}
 \label{eq:diagonal}
 (\alpha\otimes x)\cdot \zeta = (-1)^{|\zeta||x|} (\alpha\zeta\otimes x + \alpha\otimes \zeta x).   
 \end{equation}
 The starting point for the remainder of this work is that the map \eqref{eq:eta-defn} is equivariant with respect to $\eee(k)$ actions.

\begin{lemma}
\label{le:equivariance}
For $M\in \dbcat{R}$ and the $\eee(k)$-actions on $\eee(R)\otimes_k\tee(M)$ and $\eee(M)$ described above, the map $\eta(M)$ is $\eee(k)$-linear. Moreover,  $U_n(M)$ is an $\eee(k)$-submodule of $\eee(M)$ for each integer $n$.
\end{lemma}

\begin{proof}
    The equivariance of \eqref{eq:eta-defn} is a direct computation; for instance, see  \cite[Corollary~3.5]{Ferraro:2019s}. Given this it is clear that $U_n(M)$ is an $\eee(k)$-submodule of $\eee(M)$.
\end{proof}

Here is why the $\eee(k)$-linearity of $\eta(M)$ is relevant to this work.

\begin{proposition}
\label{pr:fin-gen}
Suppose that $\eee(R)$ is generated as an $\eee(k)$-module  by $\eee^{\leqslant s}(R)$ for some integer $s\ge 0$. 
Fix $M$ in $\dbcat R$ with $\inf\hh_*(M)\ge 0$. The $\eee(k)$-module $U(M)$ is generated by the subspace $U^{\leqslant s}_{s}(M)$, and hence
\[
U_n(M)=U_{s+1}(M)\quad \text{for $n\ge s+1$}\,.
\]
In particular, $\sigma(R)\le s+1$.
\end{proposition}

The integer $s$ is the top degree of the graded module 
\[
\frac{\eee(R)}{\eee(R)\cdot \eee^{\geqslant 1}(k)}\,.
\]
This is by Nakayama's Lemma for graded modules.

\begin{proof}
The task is to verify that for $n\ge s+1$ that the image of the map
\[
\eee(R)\otimes_k \tee_{<n}(M) \longrightarrow \eee(M)
\]
is in the  $\eee(k)$-submodule  generated by $U^{\leqslant s}_{s}(M)$. Give our hypothesis on $\eee(R)$, it suffices to verify that for elements $\alpha\in \eee^{\leqslant s}(R)$ and $x\in \tee_{n-1}(M)$ the element
\[
\eta(M)(\alpha\zeta \otimes x) \quad \text{ is in } \quad U^{\leqslant s}_{s}(M)\cdot \eee(k) \quad \text{for each $\zeta$ in $\eee(k)$}.
\]
Suppose $|\zeta|=0$. When $n=s+1$ clearly $\eta(M)(\alpha\otimes x)$ is in $U^{\leqslant s}_{s}(M)$, and if
$n\ge s+2$, then $\eta(M)(\alpha\otimes x)=0$ for degree reasons.  In either case, the desired inclusion holds.

We can thus assume $|\zeta|\ge 1$, and then that $\zeta$ is primitive, for $\eee(k)$ is primitively generated. Then from Lemma~\ref{le:equivariance} and \eqref{eq:diagonal} we get
\[
\eta(M)(\alpha\zeta\otimes x)  = \eta(M)(\alpha\otimes x) \cdot \zeta - \eta(M)(\alpha\otimes \zeta x) =  - \eta(M)(\alpha\otimes \zeta x) \,. 
\]
Here again we have used the fact that $\eta(M)(\alpha\otimes x)=0$ for degree reasons. Since $|\zeta x|<|x|$, an induction on $n$ gives the desired result.
\end{proof}

The preceding result means that when the $\eee(k)$-module $\eee(R)$ is finitely generated, for any $M$ in $\dbcat R$, the $\eee(k)$-submodule of $\eee(M)$ consisting of unstable elements is finitely generated; moreover there is a bound on the degrees of the generators independent of $M$. So Lescot's question~\ref{qu:lescot} leads one to ask: 

\begin{question}
\label{qu:fin-gen}
For which local rings $R$ is $\eee(R)$ finitely generated as a $\eee(k)$-module?
\end{question} 

While we do not know any local rings for which the finite generation property fails,  it seems likely that this is only for lack of looking hard enough for counterexamples. In the rest of the section, we record various families of rings for which the question above, and hence also Lescot's question~\ref{qu:lescot}, has a positive answer. The one below recovers \cite[Proposition~3.3]{Lescot:1986}.

\begin{corollary}
\label{co:gorenstein}
For any Gorenstein local ring $R$, one has $\sigma(R)=\dim R +1$.    
\end{corollary}

\begin{proof}
Proposition~\ref{pr:fin-gen} applies with $s=\dim R$.
\end{proof}

The next result recovers \cite[Proposition~3.4]{Lescot:1986}; we sketch a proof, which is different from the one in \emph{op.~cit.}, for it suggests a way to identify other classes of rings $R$ where $\sigma(R)$ is finite. For another proof of the finiteness of $\sigma(R)$ see Proposition~\ref{pr:ggolod}. 

\begin{corollary}
    \label{co:golod}
If the local ring $R$ is Golod, but not regular, $\sigma(R)\le \mathrm{edim}\, R$. 
\end{corollary}

\begin{proof}
Let $K$ denote the Koszul complex on a minimal generating set for the maximal ideal of $R$, viewed as a dg (= differential graded) $R$-algebra. Adjunction and self-duality of $K$ yields isomorphisms
\[
\ext_R(k,R) \cong \ext_K(k,\Hom_R(K,R)) \cong \ext_K(k,\Sigma^e K)\,,
\]
where $e=\mathrm{edim}\, R$. The map of dg algebras $R\to K$ induces a map of graded $k$-algebras $\ext_K(k,k)\to \ext_R(k,k)$, and the isomorphisms above are compatible with the action of $\ext_K(k,k)$. Thus, given Proposition~\ref{pr:fin-gen}, it suffices to prove that $\ext_K(k,K)$ is generated, as a module over $\ext_K(k,k)$, by $\ext^{\leqslant -1}_K(k,K)$. 

Since the ring $R$ is Golod, one has a quasi-isomorphism of dg algebra $K\simeq \Lambda$, where $\Lambda =k\ltimes V$ where $V=\hh_{\geqslant 1}(K)$; see \cite[Theorem~2.3]{Avramov:1986s}. It thus suffices to verify that the $\ext_{\Lambda}(k,k)$-module $\ext_{\Lambda}(k,\Lambda)$ is generated by its components in degree $\le -1$. Consider the exact sequence of graded $\Lambda$-modules
\[
0\longrightarrow V \longrightarrow \Lambda \longrightarrow k \longrightarrow 0\,.
\]
Since $R$ is not regular $V\ne 0$ and the induced map $\ext_{\Lambda}(k,\Lambda)\to \ext_{\Lambda}(k,k)$ is zero; this can be proved by arguing as in the proof of \cite[Theorem~2.4]{Avramov/Iyengar:2013s}. The exact sequence above thus induces the exact sequence 
\[
0\longrightarrow \Sigma^{-1} \ext_{\Lambda}(k,k) \longrightarrow  \ext_{\Lambda}(k,V) \longrightarrow  \ext_{\Lambda}(k,\Lambda)\longrightarrow 0
\]
of graded $\ext_{\Lambda}(k,k)$-modules. Since the $\Lambda$ action on $V$ factors through the augmentation $\Lambda\to k$, one has an isomorphism
\[
\ext_{\Lambda}(k,V) \cong \ext_{\Lambda}(k,k)\otimes_k V
\]
of $\ext_{\Lambda}(k,k)$-modules. It remains to note that $V^i\cong \hh_{-i}(K)=0$ for $i \ge 0$.
\end{proof}

\begin{remark}
    \label{re:tensor-product}
    Let $k$ be a field and $R,S$ local supplemented $k$-algebras; thus, $k$ is the residue field of $R$ and $S$, and the surjective maps $R\to k$ and $S\to k$ are $k$-linear. For finitely generated modules $M$ and $N$ over $R$ and $S$, respectively, one has a natural isomorphism
    \[
    \ext_R(k,M)\otimes_k \ext_S(k,N)\xrightarrow{\ \cong\ } \ext_{R\otimes_kS}(k,M\otimes_k N)
    \]
    and this map is compatible with the isomorphism of graded $k$-algebras
    \[
    \ext_R(k,k)\otimes_k \ext_S(k,k)\xrightarrow{\ \cong\ } \ext_{R\otimes_kS}(k,k)\,.
    \]
    See, for instance, \cite[Chapter XI, Theorem~3.1]{Cartan/Eilenberg:1956a}. It follows that if Question~\ref{qu:fin-gen} has a positive answer for $R$ and $S$, then it also does for $R\otimes_kS$.
\end{remark}

\subsection*{Finite linearity defect}
Roughly speaking, a finitely generated  $R$-module $M$ is said to have \emph{finite linearity defect} if in the minimal free resolution of $M$ over $R$, all differentials can be eventually represented by matrices of linear forms. It is proved in \cite{Herzog/Iyengar:2005} that in this case $\ext_R(M,k)$ is finitely generated as a left module over the $\eee(k)$. This leads to the following result.

\begin{corollary}
\label{co:linearity-defect}
Let $R$ be a Cohen-Macaulay ring with a canonical module. If the linearity defect of the canonical module is finite, then $\sigma(R)$ is finite as well.
\end{corollary}

\begin{proof}
Let $\omega_R$ denote a canonical module for $R$. When the linearity defect of $\omega_R$ is finite, then $\ext_R(\omega_R,k)$ is finitely generated as a left module over $\eee(k)$, and hence $\ext_R(k,R)$ is finitely generated as a right module over $\eee(k)$, by local duality. This implies $\sigma(R)$ is finite.
\end{proof}

This leads to the following question.

\begin{question}
Suppose that the ring $R$ is Cohen-Macaulay and has a canonical module. When is the linearity defect of the canonical module finite?
\end{question}

 Many Koszul algebras $R$ have the property that every finitely generated $R$-module  has finite linearity defect. These are the \emph{absolutely Koszul algebras}. Some, but not all, Veronese subrings of polynomial rings are absolutely Koszul; see \cite[Section~5]{Conca/Iyengar/Hop/Romer:2015}, so Corollary~\ref{co:ggolod} implies a positive answer to Question~\ref{qu:lescot} for these rings. Next we prove by an entirely different method that $\sigma(R)$ is finite for all Veronese subrings.

\subsection*{Veronese subrings}
Let $S$ be a standard graded $k$-algebra and $M$ a finitely generated graded $S$-module. Fix an integer $c\ge 0$. For integers $r$ in $[0,c-1]$ set 
    \[
    V_{c,r}(M)\coloneqq \bigoplus_{j\in \bbZ} M_{cj+r}\,.
    \]
 In particular, $S^{(c)}\coloneqq V_{c,0}(S)$ is a $k$-algebra of $S$, called the $c$'th \emph{Veronese subalgebra} of $S$. Each $V_{c,r}(M)$ is a finitely generated $S^{(c)}$-module; we may speak of these as Veronese summands of $M$ as an $S^{(c)}$-module. As a module over $S^{(c)}$ one has a decomposition
\[
M = V_{c,0}(M)\oplus \cdots \oplus V_{c,c-1}(M)\,.
\]
The following result is contained in \cite[Proposition~2.2]{Polishchuk/Positselski:2005}.

\begin{proposition}
\label{pr:veronese-linear}
With notation as above, when $S$ is Koszul, if the $S$-module $M$ has a linear resolution, so do the $S^{(c)}$-modules $V_{c,r}(M)$ for $0\le r\le c-1$. \qed
\end{proposition}

This means that when $S$ is a standard graded $k$-algebra that is Cohen-Macaulay and the $S$-module $\omega_S$ has a linear resolution, then for any integer $c$, the Veronese subalgebra $R\coloneqq S^{(c)}$, that is known to be Koszul and Cohen-Macaulay, has the property that $\omega_R$ has a linear resolution. This is because $\Hom_R(S,\omega_R)=\omega_S$, and since $R$ is a direct summand of $S$ as an $R$-module, $\omega_R$ is a direct summand of $\omega_S$ as an $R$-module; in fact the former is a Veronese summand of the latter. Thus $\sigma(S^{(c)})$ is finite, by Corollary~\ref{co:linearity-defect}. Here is a noteworthy special case. 

\begin{corollary}
Let $k$ be a field and $S=k[x_1,\dots,x_d]$, the polynomial ring over $k$ in indeterminates $x_1,\dots,x_d$. Then $\sigma(S^{(c)})<\infty$ for any integer $c\ge 1$.  \qed
\end{corollary}

\subsection*{Double Ext modules}
Let  $R$ be a local ring such that 
\[
\ext_{\eee(k)}(\ext_R(M,k),k)
\]
is finitely generated as a module over $\ext_{\eee(k)}(k,k)$ for all finitely generated $R$-modules $M$; see \cite{Backelin/Roos:1986s}. This condition implies in particular that $\ext(M,k)$ is finitely generated as (left) module over $\eee(k)$ for all $M$ in $\dbcat{R}$. By duality, the latter condition is equivalent to: $\ext_R(k,N)$ is finitely generated as a (right) module over $\eee(k)$ for each $N$ in $\dbcat R$. In particular, $\eee(R)$ is finitely generated. Thus Lescot's question has a positive answer. 
    
    \begin{proposition}
    \label{pr:Roos-rings}
    Let $R$ be a local ring that is a  Golod map away from a complete intersection.
    Then $\sigma(R)$ is finite.
    \end{proposition}

    \begin{proof}
    It suffices to point to \cite[Theorem~2, ]{Backelin/Roos:1986s} that asserts that the such rings have the double Ext property discussed above.
    \end{proof}

\subsection*{Coherent Ext algebras}
Let $k$ be a field and $A=\{A^i\}_{i\geqslant 0}$ a graded $k$-algebra (not necessarily commutative) such that $A^0=k$ and $\rank_k A^i$ finite for each $i$. 
A graded $A$-module $E$ is said to be \emph{coherent} if it is finitely generated, and each finitely generated $A$-submodule of $E$ is finitely presented.
By default, an module means a right module. In the literature, this property is usually called  \emph{graded} coherence but since we only ever deal with graded objects, we drop the adjective ``graded". When $E$ is coherent, so is the $A$-module $E(n)$ for any integer $n$, where $E(n)^i=E^{i-n}$ and $x\cdot a= xa$ for $x\in E(n)^i$ and $a\in A^j$.

The ring $A$ is \emph{right coherent} if $A$ viewed as a right $A$-module is coherent. The subcategory of the category of graded $A$-module consisting of coherent modules is abelian.  See \cite[\href{https://stacks.math.columbia.edu/tag/05CU}{Tag 05CU}]{stacks-project} for proofs of this claim in the ungraded case; the same arguments carry over to our context. 

Our focus is on the coherence of the graded $k$-algebra $\ext_R(k,k)$. Since this is a Hopf algebra, the category of right modules is equivalent to the category of left modules, so $\ext_R(k,k)$ is right coherent if and only if it is left coherent. For this reason,  we speak of the coherence of this algebra without specifying a side.

The result below extends \cite[Theorem~1']{Roos:1978s}, which deals with the case where $M$ is a module. We sketch a proof, which differs from the one in \emph{op.~cit.}. 

\begin{proposition}
\label{pr:coherent}
    Let  $R$ be a local ring such that the graded $k$-algebra $\eee(k)$ is coherent. For any $M$ in $\dbcat R$, the right $\eee(k)$-module $\eee_R(M)$ is coherent. In particular, $\eee(R)$ is finitely generated and hence $\sigma(R)$ is finite.
\end{proposition}

\begin{proof}
 Let $K$ be the Koszul complex on a finite generating set for $\fm$, the maximal ideal of $R$, and set $K^M=M\otimes_RK$. Since $K$ is a finite free $R$-complex whose differential satisfies $d(K)\subseteq \fm K$, there is an isomorphism 
\[
\eee(K^M)\cong \eee(M)\otimes_k (k\otimes_RK) 
\]
of graded right $\eee(k)$-modules, where the action of $\eee(k)$ on the right-hand side is through $\eee(M)$. As $k\otimes_RK$ is a nonzero graded $k$-vectorspace, it follows that $\eee(M)$ is a direct summand $\eee(K^M)$ as an $\eee(k)$-module. Thus it suffices to verify that the latter is coherent as an $\eee(k)$-module.  Since $M$ is in $\dbcat R$ the $R$-module $\hh(K^M)$ has finite length. Thus replacing $M$ by $K^M$ we can assume $\hh(M)$ has finite length, so that $M$ is in the thick subcategory of $\dbcat R$ generated by $k$.

Since the category of graded $\eee(k)$-modules is abelian, a simple argument shows that the subcategory of $\dbcat R$ consisting of $R$-complexes $X$ with the property that the $\eee(k)$-module $\eee_R(X)$ is coherent is thick.  It contains $k$, by hypothesis, and hence also $M$, by the discussion in the preceding paragraph.
\end{proof}

Proposition~\ref{pr:coherent} raises the following:

\begin{question}
\label{qu:coherence}
 For which local rings $R$ is the graded $k$-algebra $\eee(k)$ coherent?
\end{question}

This question is discussed in Roos' article~\cite{Roos:1982b}, in the broader context of $\lambda$-dimension of $\eee(k)$.  Applying results from \cite{Moore:2009a},  \cite{Ananthnarayan/Avramov/Moore:2012}, and \cite{Choo/Lam/Luft:1973} one gets:

\begin{lemma}
Let $R$ and $S$ be local algebras with a common residue field $k$. The fiber-product $R\times_kS$ has a coherent Ext-algebra if, and only if, so do $R$ and $S$. \qed
\end{lemma}

We know that not every local ring has the desired coherence property. 

\begin{example}
\label{ex:not-coherent}
Let $R$ be standard graded $k$-algebra that is Koszul and satisfies $\rank_kR<\infty$; thus the global dimension of $\eee(k)$ is finite. Hence if $\eee(k)$ is coherent, then $R$ is absolutely Koszul; see, for instance, the proof of \cite[Theorem~6.2.1]{Gelinas:2018}. Here is an argument: if $\eee(k)$ is coherent, then for each finitely generated $R$-module $M$ the left $\eee(k)$-module $\ext_R(M,k)$ has a finite free resolution, and hence the linearity defect of $M$ is finite. 

In summary, if $R$ is finite dimensional, standard graded $k$-algebra that is Koszul but not absolutely Koszul, then $\eee(k)$ is not coherent.  Here is such a ring: 
\[
R\coloneqq S\otimes_k S \quad\text{where}\quad S=\frac{k[x,y,z]}{(x,y,z)^2}\,.
\]
It is Koszul, being a monomial ring defined by quadratic relations, but  not absolutely Koszul, because there exist finitely generated modules over this ring whose Poincar\'e series is transcendental; see \cite{Roos:2005s}, and also the introduction in \cite{Herzog/Iyengar:2005}.

However, the ring $S$ is Golod and hence $\ext_S(k,S)$ is finitely generated over $\ext_S(k,k)$; see proof of Corollary~\ref{co:golod}, or Proposition~\ref{pr:ggolod}. Thus $\ext_R(k,R)$ is finitely generated over $\ext_R(k,k)$, by Remark~\ref{re:tensor-product}, and hence $\sigma(R)$ is finite.
\end{example}

On the other hand, \cite[Corollary~2]{Roos:1982b} identifies one class of local rings to answer Question~\ref{qu:coherence}: Golod rings. This recovers Corollary~\ref{co:golod}, though not the bound on $\sigma(R)$. Next we establish a generalization Roos' result, thereby identifying a much larger family of rings whose Yoneda ext-algebra is coherent.

\subsection*{Generalized Golod rings}
 The notion of a generalized Golod ring is  introduced in \cite{Avramov:1994a}. This class contains Golod rings, but much more; see Corollary~\ref{co:ggolod} below.
The coherence of the Ext-algebra of generalized Golod rings was stated already in \cite[Section 1.7]{Avramov:1994a}, but without  proof, so we supply it.

\begin{proposition}
\label{pr:ggolod}
 If a local ring $R$ is generalized Golod, then the right $\eee(k)$-module $\eee(M)$ is coherent  for $M$ in $\dbcat R$. In particular, $\sigma(R)<\infty$.
\end{proposition}

\begin{proof}
Given Proposition~\ref{pr:coherent} it suffices to verify that the graded $k$-algebra $\eee(k)$ is coherent. The defining property of generalized Golod rings involves the homotopy Lie algebra $\pi^*(R)$ of a local ring $R$; see \cite[1.7]{Avramov:1994a}, or \cite[\S10.2]{Avramov:1998a}. The crucial property of $\pi^*(R)$ is that its universal enveloping algebra is $\eee(k)$. In particular, the later is Hopf $k$-algebra.

To say that $R$ is generalized Golod is to say that for some integer $s$, the graded Lie subalgebra $\pi^{\geqslant s}(R)$ is free, that is to say, its universal enveloping algebra is the tensor algebra. One thus gets an exact sequence of graded Hopf $k$-algebras 
\[
1\longrightarrow T \longrightarrow \eee(k) \longrightarrow U\longrightarrow 1
\]
where $T$ is a tensor algebra on $\pi^{\geqslant s}(R)$, which is a graded $k$-vectorspace of finite rank in each degree, and $U$ is the restricted universal enveloping algebra of a finite dimensional graded Lie algebra, namely $\pi^*(R)/\pi^{\geqslant s}(R)$. Since the vectorspace $\pi^{<s}(R)$ is finite dimensional, the ring $U$ is noetherian; this follows from the Poincar\'e-Birkhoff-Witt theorem; see, for example, \cite[Corollary~2.3.8]{Dixmier:1996s}. It then follows from \cite[Theorem~3]{Roos:1982b} that $\eee(k)$ is coherent, as desired.
\end{proof}

\begin{corollary}
\label{co:ggolod}
Question~\ref{qu:fin-gen}, and hence also Question~\ref{qu:lescot}, has a positive answer any local ring $R$ such that $\mathrm{edim} R - \depth R\le 3$.    
\end{corollary}

\begin{proof}
Local rings $R$ as in the statement of the theorem are generalized Golod, by \cite{Avramov/Kustin/Miller:1988}. It remains to apply Proposition~\ref{pr:ggolod}.
\end{proof}

We refer to \cite{Avramov:1994a} for other examples of generalized Golod rings; this class includes rings that are a Golod map away from a complete intersection, so we get another proof of Proposition~\ref{pr:Roos-rings}.

\section{Bass series}
\label{se:BassSeries}
In this section we return to the question of computing Bass series of syzygy modules. As before, let $(R,\fm,k)$ be a local ring and $M$ an $R$-complex in $\dbcat R$.  The \emph{Bass series} and \emph{Poincar\'e series} of $M$ are the generating series 
\[
\bass_R^M(t)\coloneqq \sum_{i\in\bbZ}\rank_k\ext^i_R(k,M)t^i\quad\text{and}\quad \betti^R_M(t) \coloneqq \sum_{i\in\bbZ}\rank_k\tor^R_i(k,M)t^i\,,
\]
of the Bass numbers and the Betti numbers of $M$, respectively. These are formal Laurent series because $\eee^i(M)=0$ for $i<\depth_RM$, and $\tee_i(M)=0$ for $i<\inf\hh_*(M)$.
Given a formal Laurent series $p(t)=\sum_{i\in\bbZ} p_it^i$ and integer $n$, we write $[p(t)]_{n-1}$ for the polynomial $\sum_{i<n}p_it^i$.

Here is the stated expression for the Bass series of $\syz^nM$; since $U(M)=0$ if and only if $W(M)=0$, this result is contained in \cite[Theorem~2.2]{Lescot:1986}. 

\begin{theorem}
\label{th:bass-series-formula}
Fix an $R$-complex $M$ in $\dbcat R$ with $U(M)=0$. For any integer $n\ge \inf\hh_*(M)+1$ the Bass series of $\syz^n_R(M)$ is given by
    \[
    \bass_R^{\syz^n(M)}(t) =  t^{n}\bass_R^M(t) + t^{n-1}[\betti_R^M(t^{-1})]_{n-1}  \cdot \bass_R(t)\,.
    \]
\end{theorem}

\begin{proof}
As before, we write $\eee(M)=\ext_R(k,M)$ and $\tee(M)=\tor^R(k,M)$. Since $U(M)=0$ one gets that $U_n(M)=0$, and hence Lemma~\ref{le:Un-description} yields an  exact sequence sequence of graded $k$-vector-spaces 
    \[
0\longrightarrow \eee(M) \longrightarrow
        \eee(\Sigma^n \syz^n(M)) \longrightarrow
            \Sigma \eee(R)\otimes_k \tee_{<n}(M) \longrightarrow 0\,.
    \]
Since $\bass_R^{\Sigma^n \syz^n(M)} =t^{-n}\bass_R^{\syz^n(M)}$  the sequence above gives
    \begin{align*}
    t^{-n}\bass^{\syz^n(M)}(t)
        &= \bass^M(t) + t^{-1}\big[\betti^M(t^{-1})\big]_{n-1} \cdot \bass_R(t)\,.
    \end{align*}
Multiplying through with $t^n$ gives the desired result. 
\end{proof}

Given the exact sequence in Lemma~\ref{le:Un-description}, even if $U(M)\ne 0$, one can derive an expression for the Bass series of $\syz^nM$ in terms of the Bass series and Poincar\'e series of $M$, the Bass series of $R$, and a correction term involving the Hilbert series of $U_n(M)$. In this way one can recover, \cite[Theorem~2.2]{Lescot:1986}.

\subsection*{Curvature}
Following Avramov~\cite{Avramov:1996}, the \emph{injective curvature} of an $R$-complex $M$ in $\dbcat R$ is the reciprocal of the radius of convergence of its Bass series:
\[
\injcurv M \coloneqq \limsup_n \sqrt[n]{\rank_k\ext^n_R(k,M)}\,.
\]
One has $\injcurv M\le \injcurv k<\infty$; see \cite[Proposition~2]{Avramov:1996}. Here is a direct consequence of Theorem~\ref{th:bass-series-formula}. The last part is from Corollary~\ref{co:summands}. 

\begin{corollary}
Fix an $R$-complex $M$ in $\dbcat R$ and set $s=\sup\hh_*(M)$. If  $U(M)=0$, then for each integer $n\ge s+1$ there are equalities
\[
\injcurv \syz^n M = \injcurv \syz^{s+1}M =  \max\{\injcurv M, \injcurv R\}\,.
\]
Moreover, for $n\ge s$ any nonzero direct summand $N$ of  $\syz^nM$ has infinite projective dimension and infinite injective dimension.\qed
\end{corollary}

\bibliographystyle{amsplain}

\newcommand{\noopsort}[1]{}
\providecommand{\bysame}{\leavevmode\hbox to3em{\hrulefill}\thinspace}
\providecommand{\MR}{\relax\ifhmode\unskip\space\fi MR }

\providecommand{\MRhref}[2]{%
  \href{http://www.ams.org/mathscinet-getitem?mr=#1}{#2}
}
\providecommand{\href}[2]{#2}

 \end{document}